\documentclass{amsart}[12pt, papera4]
\usepackage{tikz}
\usepackage{amssymb}
\usepackage{graphicx}
\usepackage{amsfonts}
\usepackage[all,2cell]{xy} \UseAllTwocells \SilentMatrices
\usepackage{graphicx}
\usepackage{amsthm}
\usepackage{amstext}
\usepackage{amsmath}
\usepackage{amscd}
\usepackage[mathscr]{eucal}
\usepackage{url}

\newtheorem*{theorem*}{Theorem}
\newtheorem{theorem}{Theorem}[section]
\newtheorem{proposition}[theorem]{Proposition}

\newtheorem{lemma}[theorem]{Lemma}

\theoremstyle{definition}
\newtheorem{definition}[theorem]{Definition}

\theoremstyle{remark}

\numberwithin{equation}{section}

\def\N{{\mathbb N}}

\def\f{{\mathscr{F}}}
\def\g{{\mathscr{G}}}
\def\u{{\mathscr{U}}}

\begin{document}
	\begin{Large}
		
		\title[Topological $(\f,\g)-$shadowing property]{Topological $(\f,\g)-$shadowing property}
		\author[S. H. Joshi]{Shital H. Joshi}
		\address{Department of Mathematics, Shree M. P. Shah Arts and Science College, Surendranagar, India\\
			Department of Mathematics, Faculty of Science, The Maharaja Sayajirao University of Baroda, Vadodara, India}
		\email{shjoshi11@gmail.com}
		\author[E. Shah]{Ekta Shah}
		\address{Department of Mathematics, Faculty of Science, The Maharaja Sayajirao University of Baroda, Vadodara, India}
		\email{shah.ekta-math@msubaroda.ac.in}

		

		\begin{abstract}
			\noindent We define the concept of $(\f,\g)-$shadowing property on uniform space and say it as a topological $(\f,\g)-$shadowing property. We show that topological shadowing, topological $(\mathscr{D},\f_t)-$shadowing, topological $(\f_t,\f_t)-$shadowing and topological $(\N,\f_t)-$shadowing are equivalent in compact chain recurrent dynamical system. We also prove that if minimal points of $f$ are not dense in $X$ then the system does not have $(\mathscr{P}(\N),\f_{ps})-$shadowing.  
		\end{abstract}
		\maketitle
		
		\section{Introduction}
		
		\noindent The shadowing property is one of the important property in the theory of discrete dynamical systems as it is close to stability of the system and also to the chaotic behavior of the system. Recently, many generalizations and variations of shadowing property are studied. One such variation is $(\f,\g)-$shadowing property defined by Oprocha \cite{Oprocha}, where $\f$ and $\g$ are families of subsets of natural numbers. In \cite{ES}, authors studied, the interrelationship between various forms of $(\f,\g)-$shadowing property. 
		
		\noindent These days the study of such types of variations and generalizations on uniform spaces has became interesting.  Ahmadi et al. introduced and studied the topological concepts of expansivity and shadowing property for dynamical systems on noncompact nonmetrizable spaces, which generalize the relevant concepts for metric spaces \cite{SA2}. In \cite{SA1}, Ahmadi introduced and studied the topological ergodic shadowing and proved that a dynamical system with topological ergodic shadowing property is topologically chain transitive, and that topological chain transitivity together with topological shadowing property implies topological ergodic shadowing. Xunxung et al. \cite{WX} proved that every dynamical system defined on a Hausdorff uniform space with topologically ergodic shadowing is topologically mixing, thus topologically chain mixing. In \cite{Das1}, Das et al. defined topological $\underline{d}-$shadowing for a continuous map on a uniform space and show that it is equivalent to topological ergodic shadowing for a uniformly continuous map with topological shadowing on a totally bounded uniform space.  In \cite{SA4}, Ahmadi et al. proved that a surjective equicontinuous map on a compact uniform space has shadowing property if and only if the space is totally disconnected. Akoijam et al. proved that an expansive homeomorphism defined on compact uniform space has topological shadowing property then it has the periodic shadowing \cite{KBM}.
		
		\noindent In Section 1, we discuss the basics required for the content of the paper. The notion of $(\f,\g)-$shadowing property for metric spaces was introduced by Oprocha in \cite{Oprocha}. In Section 2, we define and study $(\f,\g)-$shadowing property on uniform spaces and name it as topological $(\f,\g)-$shadowing property. We then prove some Lammas to establish equivalence between topological shadowing, topological $(\mathscr{D},\f_t)-$shadowing, topological $(\f_t,\f_t)-$shadowing and topological $(\N,\f_t)-$shadowing for chain recurrent compact dynamical system. Finally, in Section 3, we define the syndetically proximal relation for the uniform spaces and prove that if the minimal points of $f$ are not dense in $X$, then $f$ does not exhibit topological $(\mathscr{P}(\N),\f_{ps})-$shadowing.
		
		\section{Preliminaries}
		\subsection{Subsets of $\N$}
		\noindent An infinite subset $A$  of $\N$ is said to be a \emph{co--finite} if $\N \setminus A$ is finite,
		\emph{thick} if $A$ contains arbitrary large blocks of consecutive numbers, \emph{syndetic} if  $A$ is infinite and there exists $M\in \N$ such that the gap between two consecutive integers in $A$ is bounded by $M$, \emph{piecewise syndetic} if $A=T\cap S$ where $T$ and $S$ are thick and syndetic subsets of $\mathbb{N}$ respectively and \emph{thickly syndetic} if for every $n\in \mathbb{N}$ the set of positions where length $n$ blocks begin forms a syndetic set.  A \emph{family} of subsets of $\N$ is any subset $\mathscr{F}$ of $\mathscr{P}(\N)$, the power set of $\N$, which is upward hereditary, i.e., if $B\in \mathscr{F}$ and $B\subset C\subset \N$ then $C\in \mathscr{F}$. The families of co--finite, thick, syndetic, piecewise syndetic and thickly syndetic subsets of $\N$ are denoted by $\mathscr{F}_{cf}$, $\mathscr{F}_t$, $\mathscr{F}_s$, $\mathscr{F}_{ps}$ and $\mathscr{F}_{ts}$ respectively. 
		
		\noindent The \emph{lower density} of $A\subset \N$ is denoted by $\underline{d}(A)$ and is defined by $$\underline{d}(A)=\liminf_{n\to \infty}\frac{|A\cap\{0,1,\dots,n-1\}|}{n}.$$ If we replace $\liminf$ with $\limsup$ in the above formula, we get $\overline{d}(A)$, the \emph{upper density} of $A$. If $\underline{d}(A)=\overline{d}(A)$, then the common value is said to be \emph{density} of $A$. The family of subsets of $\N$ with density $1$ is denoted by $\mathscr{D}$. The dual of a family $\mathscr{F}$ of subsets of $\N$ is denoted by $\mathscr{F}^*$ and is given by $\mathscr{F}^*=\{A\subset \N:A\cap B\neq \phi \mbox{ for all }B\in \mathscr{F}\}$ \cite{Oprocha}.

		\subsection{$(\f,\g)-$shadowing in metric space}
		\noindent Let $f:X\longrightarrow X$ be a continuous map defined on a metric space.	Fix $A\subset \mathbb{N}$ and $\delta$, $\epsilon$ $>0$. A sequence $\{x_i\}_{i=1}^{\infty}$ is a \emph{$\delta-$pseudo orbit on $A$} if $A\subset \{i:d(f(x_i),x_{i+1})<\delta\}.$ If $A=\N$, then $\{x_i\}_{i=1}^{\infty}$ is a $\delta$-pseudo orbit in classical sense. A point $y\in X$ is said to \emph{$\epsilon-$shadow (or $\epsilon-$traces)} a sequence $\{y_i\}_{i=1}^{\infty}$ on $B$ if $B\subset \{i:d(f^i(y),y_i)<\epsilon\}$. If $B=\mathbb{N}$ then $\{y_i\}_{i}^{\infty}$ is $\epsilon-$shadowed (or $\epsilon-$traced) by a point $y$. The notion of $(\f,\mathscr{G})-$shadowing was first introduced by Oprocha in \cite{Oprocha}. We recall the definition. 
		
		\smallskip
		\begin{definition}
			Consider families $\mathscr{F}$ and $\mathscr{G}$ of $\mathbb{N}$. A map $f$ is said to have \emph{$(\mathscr{F},\mathscr{G})-$shadowing property} if for every $\epsilon>0$ there is $\delta>0$ such that if $\{x_i\}_{i=1}^{\infty}$ is a $\delta-$pseudo orbit on a set $A$ in $\mathscr{F}$ then there is a point $x$ which $\epsilon-$shadows $\{x_i\}_{i=1}^{\infty}$ on a set $B$, for some $B\in \mathscr{G}$. 
		\end{definition}
		
		\noindent The notion of syndetically proximal pairs defined by Moothathu in \cite{Moothathu}. We recall the definition. \medskip 
		
		\noindent Let $X$ be a compact metric space and $f:X\to X$ be a continuous map. A pair $(x,y)\in X\times X$ is called a \emph{syndetically proximal pair} if for every $\epsilon>0$, the set $\{n\in \N:d\left(f^n(x),f^n(y)\right)<\epsilon\}$ is syndetic. 
		
		\subsection{Uniform spaces}
		\noindent Let $X$ be a non--empty set. The set $\Delta=\{(x,x):x\in X\}$ is called the \emph{diagonal} of $X$. Let $U$ and $V$ be subsets of $X\times X$. The \emph{composite of $U$ and $V$} is denoted by $U\circ V$ and is defined as $$U\circ V=\{(x,y)\in X\times X:\mbox{there exists} \; z\in X \mbox{ satisfying } (x,z)\in U \mbox{ and }(z,y)\in V\}.$$ For $A\subset X$, the inverse of $A$ is denoted as $A^{-1}$ and is given by $A^{-1}=\{(x,y)\in X\times X:(y,x)\in A\}$. If $A=A^{-1}$ then we say that $A$ is symmetric.  
		\begin{definition}
			A \emph{uniformity} for a set $X$ is a non-empty family $\u$ of subsets of $X\times X$ such that 
			\begin{itemize}
				\item[(a)] if $U\in \u$ then $\Delta\subset \u$;
				
				\item[(b)] if $U\in \u$, then $U^{-1}\in \u$;
				
				\item[(c)] if  $U\in \u$, then there exists $V\in \u$ such that $V\circ V\subset U$; 
				\item[(d)] if $U,\;V\in \mathscr{U}$ then $U\cap V\in \u$; and 
				\item[(e)] if $U\in \u$ and $U\subset V\subset X\times X$ then $V\in \u$.   
			\end{itemize}
			The pair $(X,\u)$ is called a \emph{uniform space}. The elements of $\u$ are called \emph{entourages} of the uniformity.
		\end{definition}

		\medskip
		\noindent A subfamily $\mathscr{B}$ of a uniformity $\u$ is a \emph{base} for $\u$ if each member of $\u$ contains a member of $\mathscr{B}$. The \emph{topology $\tau$ of the uniformity $\u$}, or the \emph{uniform topology}, is the family of all subsets $T$ of $X$ such that for each $x$ in $T$ there is $U\in \u$ such that $U[x]\subset T$, where $U[x]=\{y\in X:(x,y)\in U\}$. The set $U[x]$ is called the \emph{cross section of $U$} at point $x\in X$. For $A\subset X$, the set $U[A]=\bigcup_{x\in A}U[x]$. It is easy observe that $\left(\bigcap_{U\in \mathscr{U}} U \right)[x]=\bigcap_{U\in \mathscr{U}} U[x].$ A uniform space $(X,\u)$ is \emph{Hausdorff} if and only if $\bigcap_{U\in \u}U=\Delta.$ It is known that a uniform space $(X,\mathscr{U})$ is said to be \emph{totally bounded} if for each $U\in \mathscr{U}$ there is a finite set $F\subset X$ such that $U[F]=\bigcup_{y\in F}(U[y])=X$. For entourage $V$, $V\cap V^{-1}$ is always symmetric. Therefore without loss of generality we can assume that for any entourage $U$ we can say that there exists a symmetric entourage $V$ such that $V\circ V\subset U$. Further, it can be observed that for given entourage $U$ and for given $n\in \N$, there exists a symmetric entourage $V$ such that $V^n=V\circ V\circ \dots \circ V \subset U$.
		
		\smallskip
		\noindent Let $(X,\u)$ and $(Y,\mathscr{V})$ be two uniform spaces. A map $f:X\longrightarrow Y$ is said to be \emph{uniformly continuous} if for each $V\in \mathscr{V}$, the set $\{(x,y):(f(x),f(y))\in V\}$ is a member of $\u$. Every continuous map on a compact uniform space to a uniform space is uniformly continuous \cite{Kelley}.
		
		\smallskip
		\noindent For a uniform space $(X,\mathscr{U})$ set $2^X=\{A\subseteq X:A \mbox{ is a non--empty compact set}\}$. The uniformity generated by the sets of the form $2^U=\{(A,B)\in 2^X\times 2^X:A\subset U[B] \mbox{ and }B\subset U[A]\}$, where $U\in \mathscr{U}$, on $2^X$ is denoted by $2^{\u}$. The topology induced by this uniformity is equivalent to Vietoris topology on $2^X$ \cite{Mi}. Michael proved that a uniform space $X$ is compact if and only of $2^X$ is compact \cite{Mi}.
	
		\subsection{Dynamical System}
		\medskip
		\noindent A dynamical system is a pair $(X,f)$, where $X$ is a compact uniform space and $f:X \longrightarrow X$ is a continuous onto map. A point $x\in X$ is called a  \emph{periodic  point} of $f$ with prime period $m$ if $m$  is the smallest positive integer satisfying $f^m(x)=x$. If $x$ is a periodic point of prime period $1$ then it is called a \emph{fixed point}. For $x\in X$ and $U,\;V\subset X$, let $N_f(x,U)=\{n\in \mathbb{N}:f^n(x)\in U\}$ and $N_f(U,V)=\{n\in \mathbb{N}:f^n(U)\cap V\neq \phi\}.$ A point $x\in X$ is said to be \emph{minimal} if $N_f(x,U)$ is syndetic for every open set $U$ containing $x$. The set of minimal points of $f$ is denoted by $M(f)$. A point $x\in X$ is said to be \emph{non--wandering} if $N_f(U,U)\neq \phi$ for every open set $U$ containing $x$. The set of all non--wandering points of $f$ is denoted by $\Omega(f)$.
		\medskip

		\medskip
		\noindent A map $f$ is said to be
		\emph{transitive} if $N_f(U,V)\neq \phi$ for any pair of non-empty open subset $U,\; V$ of $X$; \emph{totally transitive} if $f^n$ is transitive for any $n\in \mathbb{N}$; \emph{weakly mixing} if $f\times f$ is transitive; \emph{mixing} if $N_f(U,V)$ is co--finite for any pair of non-empty open subset $U,\; V$ of $X$. Note that, if $f$ is transitive then for any pair $D$ and $E$ of entourages and for any pair $x$, $y$ of points in $X$ there exists a positive integer $n\geq 1$ such that $f^n(D[x])\cap E[y]\neq \phi $ \cite{KBM}. Let $(X,f)$ and $(Y,g)$ be two dynamical systems. A surjective map $\pi:X\longrightarrow Y$ is said to be a \emph{factor map} if $\pi \circ f=g \circ \pi$.
		
		\medskip 
		\noindent The notion of shadowing property on uniform spaces was first defined in \cite{Das2}. We recall the definition. Let $(X,f)$ be a dynamical system. Let $D$ and $E$ be entourages. 
		A sequence $\{x_i\}_{i\in \mathbb{N}}$ is said to be \emph{$D-$pseudo orbit} if $\left(f(x_i),x_{i+1}\right)\in D$, for all $i\in \mathbb{N}$. A $D-$pseudo orbit $\{x_i\}_{i\in \mathbb{N}}$ is said to be \emph{$E-$traced} by $y$ if $\left(f^i(y),x_i\right)\in E$, for all $i\in \mathbb{N}$. 
		
		\begin{definition}
		A map $f$ is said to have \emph{topological shadowing} property if for every entourage $E$ there is an entourage $D$ such that every $D-$pseudo orbit is $E-$traced by some point in $X$.	
		\end{definition}	

		\medskip		
		\noindent Let $D$ be an entourage. A \emph{$D-$chain} is a finite set $\{x_0,x_1,\dots,x_n\}$ such that $\left(f(x_i),x_{i+1}\right)\in D$, for $0\leq i\leq n-1$. A map $f$ is said to be
		\emph{topologically chain transitive} if for every $x,y\in X$ and entourage $D$, there is a $D-$chain from $x$ to $y$ and
		\emph{topological chain mixing} if for every $x, y\in X$ and entourage $D$, there exists $N\in \N$ such that for any $n\geq N$, there are $D-$chains of length $n$ from $x$ to $y$ \cite{SA3}. A point $x\in X$ is called \emph{chain recurrent} if for every entourage $D$, there is a $D-$chain from $x$ to $x$. The set of all chain recurrent points of $f$ is denoted by $CR(f)$. For completion we give the proof of following Proposition as we could not find in literature.
		
		\begin{proposition}
			Let $(X,f)$ be a compact dynamical system. Suppose $f$ has topological shadowing. Then $\Omega(f)=CR(f)$. 
		\end{proposition}
		\begin{proof}
			Suppose $x\in \Omega(f)$. We show that $x\in CR(f)$. Let $E$ be an entourage. By uniform continuity of $f$, there exists a symmetric entourage $E'$ such that $E'^2\subset E$ satisfying $(p,q)\in E'$ implies $(f(p),f(q))\in E$. Since $x\in \Omega(f)$ it follows that there exists $n\in \N$ such that $f^n(E'[x])\cap E'[x]\neq \phi$. Let $y\in f^n(E'[x])\cap E'[x]$. Then there exists $z\in E'[x]$ such that $(x,f^n(z))\in E'$ and $(x,z)\in E'$. Consider a finite sequence $\{x,f(z), f^2(z),\dots, f^{n-1}(z),x\}$. Then it is an $E-$chain from $x$ to $x$. Hence, $x\in CR(f)$.
			
			\smallskip
			\noindent Suppose $x\in CR(f)$. We show that $x\in \Omega(f)$. Let $E$ be an entourage and $D$ obtain by topological shadowing of $f$. Since $x\in CR(f)$ it follows that there exists a $D-$chain $\{x=x_1,x_2,x_3,\dots,x_n=x\}$ from $x$ to $x$. Consider a sequence $\{z_i\}=\{x_1,x_2,\dots,x_{n-1},x_1,x_2,\dots\}$. Then $\{z_i\}$ is a $D-$pseudo orbit. Therefore there exists $z\in X$ such that $\{z_i\}$ is $E-$traced by $z$. Equivalently, $(f^i(z),z_i)\in E$, for all $i\in \N$. In particular, $(z,x)\in E$ and $(f^n(z),x)\in E$. Hence $f^n(E[x])\cap E[x]\neq \phi$. Thus, $x\in \Omega(f)$.
		\end{proof}

		\section{$(\f,\g)-$shadowing property}
		\noindent In this section we study the notion of $(\f,\g)-$shadowing on uniform space. Let $(X,f)$ be a dynamical system. Fix $A,\; B\subset \mathbb{N}$ and entourages $D$, $E$. A sequence $\{x_i\}_{i=0}^{\infty}$ is a \emph{$D-$pseudo orbit on $A$} if $A\subset \{i:(f(x_i),x_{i+1})\in D\}.$ A sequence $\{y_i\}_{i=0}^{\infty}$ is \emph{$E-$shadowed (or $E-$traced)} by $y\in X$ on $B$ if $B\subset \{i:(f^i(y),y_i)\in E\}$. 	
		\begin{definition}
			Consider families $\mathscr{F}$ and $\mathscr{G}$ of $\mathbb{N}$. A map $f$ is said to have \emph{topological $(\mathscr{F},\mathscr{G})-$shadowing property} if for every entourage $E$ there is an entourage $D$ such that if $\{x_i\}_{i=1}^{\infty}$ is a $D-$pseudo orbit on a set $A$, where $A \in \mathscr{F}$, then there is a point $x\in X$ such that $\{x_i\}_{i=1}^{\infty}$ is $E-$shadowed by $x$ on a set $B$, for some $B\in \mathscr{G}$. 
		\end{definition}

		\noindent The following Lemma is obvious from the definition of topological $(\f,\g)-$shadowing property.
		\begin{proposition}\label{L7}
			Suppose $\f_0$, $\f_1$, $\g_0$, $\g_1$ are families of subsets of $\N$ with $\f_1\subset \f_0$ and $\g_0\subset \g_1$. If $(X,f)$ has topological $(\f_0,\g_0)-$shadowing then it also has topological $(\f_1,\g_1)-$shadowing property.
		\end{proposition}
		
		\noindent In the following we obtain equivalence of topological shadowing, topological $(\N,\f_t)-$shadowing and topological $(\f_t,\f_t)-$shadowing.
		
		\begin{theorem}\label{T2}
			Let $(X,f)$ be a dynamical system such that $f$ is chain-recurrent. Then the following are equivalent:
			\begin{itemize}
				\item [(1)] $f$ has topological shadowing.
				\item[(2)] $f$ has topological $(\mathscr{D},\f_t)-$shadowing.
				\item[(3)] $f$ has topological $(\f_t,\f_t)-$shadowing.
				\item[(4)] $f$ has topological $(\mathbb{N},\f_t)-$shadowing.
			\end{itemize}	
		\end{theorem}
		
		\noindent We prove the following Lemmas required for the proof of Theorem \ref{T2}.
		
		\begin{lemma}\label{L1}
			Let $(X,f)$ be a dynamical system. For given entourage $E$ there exists an entourage $D$ such that if $\{x_0,\dots,x_n\}$ is a $D-$chain from $x_0$ to $x_0$ then there exists an $E-$chain $\{y_0,\dots,y_n\}\subset CR(f)$ from $y_0$ to $y_0$ such that $(x_i,y_i)\in E$, for all $i$, $0\leq i\leq n$.
		\end{lemma}
		\begin{proof}
			Let $E$ be an entourage. Then we show that there exists an entourage $D$ such that if $\{x_0,\dots, x_n\}$ is a $D-$chain from $x_0$ to $x_0$ then there exists an $E'-$chain $\{y_0,\dots,y_n\}\subset CR(f)$ such that $(x_i,y_i)\in E'$ for all $i$, $1\leq i\leq n$, where $E'$ is a symmetric entourage such that $E'^3\subset E$.
			
			\noindent If possible suppose there is an entourage $E'$ such that for every entourage $F$ there exists $n_F>0$ and $F-$chain $\{x_0^{(F)},\dots,x_{n_F}^{(F)}\}$ from $x_0^{(F)}$ to $x_0^{(F)}$ satisfying for every $E'-$chain $\{y_0,\dots,y_{n_F}\}\subset CR(f)$ there is $i$ satisfying $0\leq i\leq n_F$, and $(x_i^{(F)},y_i)\notin E'$. Denote $$C_F=\{x_j^{(F)}:0\leq j \leq n_F\}.$$ Then $C_F\in 2^X$. Since $X$ is compact it follows that the space $2^X$ is compact. Therefore, the infinite set $\{C_F:F\in \mathscr{U}\}$ has a limit point $C$ (say) and there exists a sequence $\{C_{F_k}\}$ such that $\{C_{F_k}\}$  converges to $C$. Let $z\in C$. Then there is a sequence $\{z_{F_k}\}$ in $C_{F_k}$ such that $z_{F_k}$ converges to $z$. This implies that $z\in CR(f)$ and hence $C\subset CR(f)$.
			
			\medskip
			\noindent Let $E''$ be an entourage such that $E''^2\subset E'$. Then by uniform continuity of $F$ there is a symmetric entourage $D$ such that $D^2\subset E''$ and $(p,q)\in D$ implies $(f(p),f(q))\in E''$. Take $k>0$ such that $F_k\subset D$, for all $i=0,\dots,n_{F_k}$. Then there exists $y_i\in C$ such that $(x_i^{F_k},y_i)\in F_k\subset D\subset E'$. But the sequence $\{y_0,y_1,\dots,y_{n_{F_k}}\}\subset C\subset CR(f)$ is also an $E'-$chain.
			This is a contradiction with the choice of the sequence $\{x_0^{(F)},\dots, x_{n_F}^{(F)}\}$. Hence the claim. 
			
			\noindent If $\{x_0,\dots,x_n\}$ is a $D-$chain from $x_0$ to $x_0$, $\{y_0,\dots,y_n\}$ is an $E'-$chain and $(x_i,y_i)\in E'$, then $\left(f(y_{n-1}),y_0\right)\in E'^3\subset E$. Hence $\{y_0,\dots,y_{n-1},y_0\}$ is an $E-$chain from $y_0$ to $y_0$ and $(x_i,y_i)\in E$ for $i=0$ to $n$.
		\end{proof}
		
		\begin{lemma}\label{L2}
			Let $(X,f)$ be a dynamical system such that $f$ is uniformly continuous. For every entourage $E$ there is an entourage $D$ such that for every thick set $T$ and $D-$pseudo orbit $\{x_i\}_{i=1}^{\infty}$ on $T$ there exists a thick set $T'\subset T$ and an $E-$pseudo orbit $\{y_i\}_{i=1}^{\infty}\subset CR(f)$ on $T'$ such that $(x_i,y_i)\in E$, for every $i\in T'$.
		\end{lemma}
		\begin{proof}
			Let $E_1$ be a symmetric entourage such that $E_1^2\subset E$, $E_2$ be provided for $E_1$ by Lemma \ref{L1} satisfying $E_2^2\subset E_1$ and $E_3$ be such that $E_3^2\subset E_2$. By uniform continuity of $f$ there is a symmetric entourage $D$ with $D^2\subset E_3$ such that if $(p,q)\in D$ then $(f(p),f(q))\in E_3$. Let $D_1$ be a symmetric entourage such that $D_1^2\subset D$. Since $X$ is compact, it is totally bounded. Therefore there is a finite set $S=\{y_1,\dots,y_s\}$ such that $$D_1[S]=\bigcup_{i=1}^{s} (D_1[y_i])=X.$$
			
			\noindent For every $i$, choose a point $p_i\in D_1[y_i]$. Let $T$ be a thick set and let $\{x_i\}_{i=1}^{\infty}$ be a $D-$pseudo orbit on $T$. For every $i\geq 1$, fix an index $j(i)$ such that $x_i\in D_1[y_{j(i)}]$ and denote $z_i=p_{j(i)}$. Then $(z_i,x_i)\in D_1^2\subset D$, for every $i\in \N$. Further, uniform continuity of $f$ implies $(f(z_i),f(x_i))\in E_3$ and therefore $\left(f(z_i),z_{i+1}\right)\in E_2$, for all $i\in T$, as $E_3^2\subset E_2$. Hence $\{z_i\}_{i=1}^{\infty}$ is an $E_2-$pseudo orbit on $T$ and $\{z_i:i\in \mathbb{N}\}\subset \{p_i:1\leq i\leq s\}$. Next, thickness of $T$ implies there exists an integer $l$, $1\leq l\leq s$, and two increasing sequences of integers $a_k<b_k<a_{k+1}-1$ such that $\lim_{k\to \infty}b_k-a_k=\infty$ and $\{a_k,a_k+1,\dots,b_k\}\subset T$ and $z_{a_k}=z_{b_k}=p_l$, for every $k$. But this implies each sequence $\{z_{a_k},\dots,z_{b_k}\}$ is $E_2-$chain from $p_l$ to $p_l$. 
			By Lemma \ref{L1}, there exists an $E_1-$chain $\{y_{a_k},y_{a_k+1},\dots,y_{b_k}\}\subset CR(f)$ such that $(y_i,z_i)\in E_1$. Put $T'=\mathbb{N}\cap \bigcup_{k=1}^{\infty}[a_k,b_k]$ and fix any point $q\in CR(f)$. For $i\notin T'$, put $y_i=q$. Note that $[a_s,b_s]\cap[a_t,b_t]=\phi$ for $s\neq t$ and so the sequence $\{y_i\}_{i=1}^{\infty}$ is an $E-$pseudo orbit on $T'$, and for every $i\in T'$ we also have $(y_i,x_i)\in E_1^2\subset E$.
		\end{proof}
		
		\begin{lemma}\label{L3}
			Let $(X,f)$ be a dynamical system. For every entourage $E$ there is an entourage $D$ such that for every $D-$pseudo orbit $\{x_i\}_{i=1}^{\infty}$ of $f$ there exist an $E-$pseudo orbit $\{y_i\}_{i=1}^{\infty}\subset CR(f)$ and $N>0$ such that $(x_i,y_i)\in E$, for every $i\geq N$.
		\end{lemma}
		\begin{proof}
			Let $E'$ be a symmetric entourage such that $E'^2\subset E$, $F$ be a symmetric entourage provided for $E'$ by Lemma \ref{L1} with $F\subset E'$ and $F'$ be a symmetric entourage such that $F'^2\subset F$. By uniform continuity of $f$ there exists a symmetric entourage $D$ with $D^2\subset F'$, such that if $(p,q)\in D$ then $(f(p),f(q))\in F'$. Let $D'$ be a symmetric entourage such that $D'^2\subset D$. Since $X$ is compact, it follows that it is totally bounded. Therefore there is a finite set $A=\{a_1,\dots a_n\}$ such that $D'[A]=X$. Take  $p_j\in D'[a_j]$, for each $j$. Let $\{x_i\}_{i=1}^{\infty}$ be a $D-$pseudo orbit. Then, for every $i\in \N$, $(f(x_i),x_{i+1})\in D$. Further, for every $i$ there is $j(i)$ such that $x_i\in D'[a_{j(i)}]$. For this $j(i)$, denote $p_{j(i)}=z_i$. Then, for $i\in \N$, $(x_i,a_{j(i)})$, $(z_i,a_{j(i)})\in D'$ as $p_{j(i)}=z_i\in D'[a_{j(i)}]$. But this implies, $(z_i,x_i)\in D$ and $(f(z_i),z_{i+1})\in F$. Now, $\{z_i:i\in \mathbb{N}\}\subset \{p_i:1\leq i\leq s\}$ implies there exists an integer $l$, $1\leq l\leq s$, and an increasing sequence $\{a_k\}_{k=1}^{\infty}$ such that $z_{a_k}=z_{a_{k+1}}=p_l$, for every $k\geq 1$. Note that the sequence $\{z_i\}_{i=a_1}^{\infty}=\{z_{a_1},z_{a_1+1},\dots,z_{a_2-1},z_{a_2},\dots,z_{a_3},\dots\}$. Therefore $\{z_i\}_{i=a_1}^{\infty}$ consists of $F-$chain from $p_l$ to $p_l$, for all $k\geq 1$.
			
			\noindent By Lemma \ref{L1}, for every $k$ there exists an $E'-$chain $\{y_{a_k},y_{a_{k}+1},\dots,y_{a_{k+1}}\}\subset CR(f)$ such that $(y_i,z_i)\in E'$. Assuming that the starting point of each constructed chain is the same, i.e., $y_{a_1}=y_{a_k}$, for every $k$. Define $\{y_i\}_{i=1}^{a_1}$ to be the last $a_1$ elements in the $E'-$chain $\{y_{a_1},y_{a_1+1},\dots,y_{a_{a_1+1}}\}$. Thus we obtain an $E-$pseudo orbit $\{y_i\}_{i=1}^{\infty}\subset CR(f)$ and $(y_i,x_i)\in E$ for every $i\geq a_1$. Taking $N=a_1$, will complete the proof.
		\end{proof}
		
		\noindent In \cite{SA4}, Ahmadi et al. defined a relation $\sim_D$ on $X$ by putting $x\sim_D y$ if there exists a $D-$chain from $x$ to $y$ for topologically chain transitive map. They proved that $\sim_D$ is an equivalence relation on $X$ and for any $x\in X$, $[x]_{\sim_D}$, the equivalence class of $x$ under $\sim_D$ is clopen set. We define $\sim_D$ on $CR(f)$. 
		
		\begin{definition}
			Let $(X,f)$ be a dynamical system. For an entourage $D$, a relation $\sim_D$ on $CR(f)$ by putting $x\sim_D y$ if there is a $D-$chain from $x$ to $y$.
		\end{definition} 
		
		\noindent The following result is similar to the Lemma 3.5 of \cite{Oprocha}. We give the proof for completion. 
		\begin{lemma}\label{6L-1}
			Let $(X,f)$ be a dynamical system. For any entourage $D$, $\sim_D$ defined on $CR(f)$ is equivalence relation and there are finitely many equivalence classes which are clopen sets.
		\end{lemma}
		\begin{proof}
			For any $x\in CR(f)$, $x\sim_D x$. Also, $x\sim_D y$ and $y\sim_D z$ implies there is a $D-$chain $\{x=x_0,x_1,\dots,x_p=y\}$ from $x$ to $y$ and $D-$chain $\{y=y_0,y_1,\dots,y_q=z\}$ from $y$ to $z$. Consider a sequence $\{x=x_0,x_1,\dots, x_p=y=y_0,y_1,\dots,y_q=z\}$, which is a $D-$chain from $x$ to $z$. Hence, $x\sim_D z$.
			
			\noindent Next, observe that if $x,\; y\in CR(f)$ are such that $\left(f(x),y\right)\in D$, then there exists a symmetric entourage $E$ with $E^3\subset D$ such that $\left(f(x),y\right)\in E$. Also, there exist $E-$chains $\{x_0,\dots,x_n\}$ and $\{y_0,\dots,y_m\}$ from $x$ to $x$ and from $y$ to $y$ respectively. Since $(f(y_{m-1}),x_1)\in E$, $(y,f(x))\in E$ and $(f(x),x_1)\in E$, it follows that $\left(f(y_{m-1}),x_1\right)\in D$. Therefore, a sequence $\{y=y_0, y_1,\dots, y_{m-1},x_1,\dots,x_n\}$ is a $D-$chain from $y$ to $x$. By induction, if there is a $D-$chain from $x$ to $y$ contained completely in $CR(f)$ then there is also a $D-$chain contained in $CR(f)$ from $y$ to $x$. Hence $y\sim_D x$. Thus $\sim_D$ is an equivalence relation on $CR(f)$.
			
			\noindent Let $x\in CR(f)$ with$D-$chain $\{x_1=x,x_2,\dots,x_n=x\}$. Then there exists entourage $E$ such that $E\subset D$ and if $z\in E[x]$ then $(f(z),x_1)$, $(f(x_{n-1}),z)\in D$. This implies that $z\in CR(f)$ because, $\{z,x_1,\dots,x_{n-1},z\}$ is a $D-$chain from $z$ to $z$. Now, $\{z,x_1,\dots,x_n=x\}$ is a $D-$chain from $z$ to $x$. Hence $z\in [x]_{\sim_D}$. This implies that $E[x]\cap CR(f)\subset [x]_{\sim_D}$. Thus, $[x]_{\sim_D}$ is an open set. Since $X$ is compact, it follows that there are only finitely many equivalence classes and they are closed too.
		\end{proof}
		
		\noindent Ahmadi et al. proved similar to the Lemma \ref{L5} for topologically chain transitive system \cite{SA4}. We state the same result on $CR(f)$.
		\noindent The proof of the Lemma \ref{L5} and Lemma \ref{L6} are analogous to the proof of Lemma 3.6 and Lemma 3.7 of \cite{Oprocha}, respectively, in the case of metric space. So, we omit the proof to avoid repetition.
		\begin{lemma}\label{L5} 
			Let $(X,f)$ be a dynamical system. Fix an entourage $D$ and let $A=[z]_{\sim_D}$ for some $z\in CR(f)$. Then there exists $k_A\geq 1$ such that for every $x\in A$ the number $k_A$ is the gcd of the lengths of all $D-$chains from $x$ to $x$ contained in $CR(f)$. 
		\end{lemma}

		\begin{lemma} \label{L6}
			Let $(X,f)$ be a dynamical system and let $E$ and $D$ be two entourages such that every $D-$pseudo orbit contained in $CR(f)$ can be $E-$traced. If $T$ is a thick set and $\{x_i\}_{i=1}^{\infty}\subset CR(f)$ is a $D-$pseudo orbit on $T$ then there exists a thick set $T'\subset T$ and $z$ such that $\left(f^j(z),x_j\right)\in E$ for every $j\in T'$.
		\end{lemma}
		
		\begin{sloppypar}
		\noindent In the following Theorem we show that if $f$ has topological shadowing then it has $(\f_t,\f_t)-$shadowing.
		\end{sloppypar}
		
		\begin{theorem}\label{T1}
			Let $(X,f)$ be a compact dynamical system. If $f$ has topological shadowing then it has topological $(\f_t,\f_t)-$shadowing.			
		\end{theorem}
		\begin{proof}
			Let $E$ be a given entourage and $E'$ be a symmetric entourage such that $E'^2\subset E$. By shadowing property there exists an entourage $D$ with $D\subset E'$ such that every $D-$pseudo orbit can be $E'-$traced. For entourage $D$, by Lemma \ref{L2} there is an entourage $F$.
			
			\noindent We complete the proof by showing that any $F-$pseudo orbit $\{x_i\}_{i=1}^{\infty}$ on a thick set $T$, for some $T\in \f_t$, is $E-$traced on a thick set by some point of $X$. Now, $\{x_i\}_{i=1}^{\infty}$ is a $D-$pseudo orbit on $T$. Therefore by Lemma \ref{L2} there is a thick set $T'\subset T$ and a $D-$pseudo orbit $\{y_i\}_{i=1}^{\infty}$ on $T'$ with $\{y_i\}_{i=1}^{\infty}\subset CR(f)$. Further, by Lemma \ref{L6}, there is a thick set $T''$ with $T''\subset T'$ and a point $z\in X$ such that $z$ $E'-$traces $\{y_i\}_{i=1}^{\infty}$ on $T''$. Therefore, for all $i\in T''$, $(f^i(z),y_i)\in E'$. Again, $T''\subset T'\subset T$ and $(x_j,y_j)\in D$, for all $j\in T'$. Since $D\subset E'$ it follows that $j\in T'$, $(x_j,y_j)\in E'$. Hence $(f^i(z),x_i)\in E'^2\subset E$. Thus $\{x_i\}_{i=1}^{\infty}$ is $E-$traced by $z$ on a thick set $T''$ and hence $f$ has topological $(\f_t,\f_t)-$shadowing property. 
		\end{proof}
				
		\smallskip
		\noindent Let $(X,f)$ be a dynamical system. Recall, a map $f$ has topological finite shadowing if for every entourage $E$ there exists an entourage $D$ such that every finite $D-$pseudo orbit in $X$ is $E-$traced by some point of $X$.
		
		\begin{lemma}\cite{KBM}\label{L10}
			Let $(X,f)$ be a dynamical system. If $f$ has topological finite shadowing, then $f$ has topological shadowing  property. 
		\end{lemma}
		
		\begin{proof}[\bf Proof of Theorem~\ref{T2}]
			
			 We show that $(1)\implies (3)\implies (2)\implies (4)\implies (1)$. By Theorem \ref{T1} $(1)\implies (3)$ follows. Further, $\N\subset \mathscr{D}\subset \f_t$. Therefore by Proposition \ref{L7}, we obtain $(3) \implies (2) \implies (4)$. Hence it only remains to show that $(4) \implies (1)$.
			
			\noindent Let $E$ be an entourage. Then there exists an entourage $D$ satisfying the definition of topological $(\N,\f_t)-$shadowing. In view of Lemma \ref{L10} it is sufficient to show that every $D-$chain is $E-$traced by some point of $X$. Let $\{x_1,\dots,x_n\}$ be a $D-$chain. The relation $\sim_D$ is an equivalence relation on $CR(f)=X$. Hence $\{x_i\}_{i=1}^{n}\subset [x_1]_{\sim_D}$. Further, there exists a $D-$chain $\{y_1,\dots, y_m\}$ from $x_n$ to $x_1$. Consider, a sequence $$\{z_i\}_{i=1}^{\infty}=\{x_1,\dots,x_{n-1},y_1,\dots,y_{m-1},x_1,\dots,x_{n-1},y_1,\dots\}.$$ Then $\{z_i\}$ is a $D-$pseudo orbit and therefore by topological $(\N,\f_t)-$shadowing there exists a point $z$ in $X$ and a thick set $T$, $T\in \f_t$ such that $$\left(f^j(z),z_j\right)\in E, \mbox{ for all } j\in T.$$
			Since $T$ contains arbitrary large blocks of consecutive integers, there is $k$ such that $$[k(n+m),k(n+m)+n]\cap \N\subset T.$$
			Put $N=k(m+n)$ and $x=f^N(z)$. Then, for $j=1,\dots,n$, $$\left(f^j(x),x_j\right)=\left(f^{N+j}(z),z_{N+j}\right)\in E.$$ Hence $\{x_i\}_{i=1}^{\infty}$ is $E-$traced by $x$.
		\end{proof}
		
		\section{Syndetically proximal relation in uniform dynamical system}
		
		\begin{sloppypar}
		\noindent In this section we give the condition under which a map do not have $(\mathscr{P}(\N),\f_{ps})-$shadowing.
		\end{sloppypar}
					
		\noindent In the following we define syndetically proximal pair on a uniform space.
		\begin{definition}
			Let $(X,f)$ be a dynamical system. A pair $(x,y)\in X\times X$ is said to be \emph{syndetically proximal pair} if for every $E\in \mathscr{U}$, the set $\{n\in \N:\left(f^n(x),f^n(y)\right)\in E\}$ is syndetic. The set of syndetically proximal pairs of $f$ is denoted by $SPR(f)$. 
		\end{definition}
		
		\noindent In the following proposition we obtain a condition for a pair $(x,y)$ to be syndetically proximal pair.
		\begin{proposition}\label{L8}
			Let $(X,f)$ be a dynamical system. Then $(x,y)\in SPR(f)$ if and only if $\{n\in \N:\left(f^n(x),f^n(y)\right)\in E\}$ is a thickly syndetic set, for every entourage $E$.
		\end{proposition}
		\begin{proof}
			Let $(x,y)\in SPR(f)$, $E$ be an entourage and $k\in \N$. Then by uniform continuity of $f$ there exists an entourage $D$ such that $$(a,b)\in D \implies \left(f^j(a),f^j(b)\right)\in E, \mbox{ for }0\leq j<k.$$
			Note that the set $\{n\in \N:\left(f^n(x),f^n(y)\right)\in D\}$ is syndetic as $(x,y)\in SPR(f)$.
			Also, $\{n\in \N:\left(f^n(x),f^n(y)\right)\in D\}\subset \{n\in \N:\left(f^{n+j}(x),f^{n+j}(y)\right)\in E, \mbox{ for } 0\leq j<k\}$. Therefore for given $k$, $\{n\in \N:\left(f^{n+j}(x),f^{n+j}(y)\right)\}\in E,\; 0\leq j\leq k\}$ is a syndetic, being a superset of a syndetic set. Hence $\{n\in \N: (f^n(x),f^n(y))\in E\}$ is a thickly syndetic set.
			
			\noindent Converse is obvious as every thickly syndetic set is syndetic.
		\end{proof}
		
		\begin{lemma}\cite{James}\label{L12}
			Let $(X,\mathscr{U})$ be a uniform space. Then with respect to the product topology on $X\times X$, the interior of entourages of $X$ form a base for the uniformity $\mathscr{U}$.
		\end{lemma}
		
		\noindent In the following Theorem, we obtain condition for a map $f$ to be syndetically proximal.
		\begin{theorem}\label{L11}
			Let $(X,f)$ be a dynamical system, where $X$ is Hausdorff uniform space. Then map $f$ is syndetically proximal if it has a fixed point which is a unique minimal subset of $X$.
		\end{theorem}
		\begin{proof}
			
			\noindent Let $E$ be an entourage and $M(f)=\{c\}$. Then $M(f\times f)=\{(c,c)\}$. We first show that $$\bigcup_{n=1}^{\infty}(f\times f)^{-n}(E^{\circ})=X\times X$$ where $E^{\circ}$ is interior of $E$. Let $(x,y)\in X\times X$. 
			Since every orbit closure contains minimal point it follows that $(c,c)\in \overline{O_{f\times f}(x,y)}$. We consider the following two cases: 
			
			\smallskip
			\noindent Case (i): Suppose $(c,c)\in O_{f\times f}(x,y)$. Then there exists $n\in \N$ such that $(c,c)=\left(f^n(x),f^n(y)\right)$. Therefore, $\left(f^n(x),f^n(y)\right)=(c,c)\in \Delta \subset E^{\circ}$.
			
			\noindent Case (ii): Suppose $(c,c)$ is a limit point of $O_{f\times f}(x,y)$. Then there exists an increasing sequence $\{n_k\}$ of natural numbers such that $$(f\times f)^{n_k}(x,y)\to (c,c), \mbox{ as } n_k\to \infty.$$ Equivalently, $$f^{n_k}(x)\to c \mbox{ and } f^{n_k}(y)\to c, \mbox{ as } n_k\to \infty.$$
			Let $F$ be a symmetric entourage such that $F\circ F\subset E$. Then by Lemma \ref{L12} there exists $k_0\in \N$ such that $f^{n_k}(x)\in F^{\circ}[c]$ and $f^{n_k}(y)\in F^{\circ}[c],$ for all $k\geq k_0$. This implies for all $k\geq k_0$, $\left(f^{n_k}(x),c\right)\in F^{\circ}$ and $\left(f^{n_k}(y),c\right)\in F^{\circ}$ and therefore $\left(f^{n_k}(x),f^{n_k}(y)\right)\in F^{\circ}\circ F^{\circ}\subset E^{\circ}$.
			
			\noindent Thus, in any case there exists $n\in \N$ such that 
			\begin{equation}\label{E1} 
				\left(f^n(x),f^n(y)\right)\in E^{\circ}
			\end{equation}
			Hence, $\bigcup_{n=1}^{\infty}(f\times f)^{-n}(E^{\circ})=X\times X$. But $X\times X$ is compact. Therefore there exists $k_E\in \N$ such that 
			\begin{equation}\label{E2}
				X\times X=\bigcup_{n=1}^{k_E}(f\times f)^{-n}(E^{\circ})=\bigcup_{n=1}^{k_E}(f\times f)^{-n}(E). 
			\end{equation}
			
			\noindent Next, for any entourage $E$, denote the set $\{n\in \N:(f^n(x),f^n(y))\in E\}$ by $N_E$. Then by Equation \ref{E1}, $N_E\neq \phi$. We complete the proof by showing that $N_E$ is a syndetic set. Let $m\in N_E$. Then $(f^m(x),f^m(y))\in E$. Also, there exists $j$, $1\leq j\leq k_E$ such that $(f^{m+j}(x),f^{m+j}(y))\in E$. This implies $m+j\in N_E$, for some $j$, $1\leq j\leq k_E$. Thus, the next element in $N_E$ is $m+r$, where $r=\min\{j:\left(f^{m+j}(x), f^{m+j}(y)\right)\in E,\; 1\leq j\leq k_E\}$. Hence $N_E$ is a syndetic set with bounded gap $k_E$.
		\end{proof}
		
		\noindent In the following Theorem, we show that $f$ does not have $(\mathscr{P}(\N),\f_{ps})-$shadowing under certain condition.
		\begin{theorem}\label{T3}
			Let $(X,f)$ be a dynamical system, where $X$ is a Hausdorff uniform space. If minimal points of $f$ are not dense in $X$ then $f$ does not have topological $(\mathscr{P}(\N),\f_{ps})-$shadowing.
		\end{theorem}
		\begin{proof}
			Let $M=M(f)$. Then $M\neq X$ as $M(f)$ is not dense in $X$. Consider the quotient space $Y=X/M$ with the factor map $\pi:X\to Y$. Define a map $g$ on $Y$ such that $\pi \circ f=g\circ \pi$. Then, $(Y,g)$ is a dynamical system. Put $p=\pi(M)$. Then $\{p\}$ is a unique minimal subset of $Y$. Therefore by Theorem \ref{L11} $g$ is syndetically proximal. Also, there exists an entourage $E$ such that $E[p]\subset U$. Put $W=\pi^{-1}(U)$. Then choice of $U$ implies that $\overline{W}\neq X$. Therefore there exists $x\in X\setminus W$ and entourage $D$ such that $D[x]\cap W=\phi$. 
			
			\noindent If possible suppose $f$ has $(\mathscr{P}(\N),\f_{ps})-$shadowing property. Consider a constant sequence $\{x,x,\dots\}$. Then $\{x,x,\dots\}$ is a $\phi-$pseudo orbit. Therefore $\{x,x,\dots\}$ is $D-$traced by some point $z$ in $X$ on a set in $\f_{ps}$, say, $F$. This implies $F=\{n\in \N:f^n(z)\in D[x]\}$. Put $y=\pi(z)$. Then $F\subset \{n:f^n(z)\notin W\}$. But $f^m(z)\notin W$ implies $g^m(y)\notin U$. Hence, $F\subset \{n:g^n(y)\notin U\}$. Again, $g$ is syndetically proximal. Therefore by Lemma \ref{L8}, $\{n:g^n(y)\notin U\}$ is thickly syndetic set. Hence $\{n:g^n(y)\notin U\}\notin \f_{ps}$ and therefore $F\notin \f_{ps}$, which is a contradiction.
		\end{proof}

		\addcontentsline{toc}{section}{\emph{Bibliography}}
			
	\end{Large}
\end{document}